\documentclass[conference]{ieeeconf}
\IEEEoverridecommandlockouts
\usepackage[noadjust,sort,compress]{cite}
\usepackage{amsmath,amssymb,amsfonts}

\usepackage[T1]{fontenc}
\usepackage{xcolor}
\definecolor{darkgreen}{RGB}{0,100,0}
\usepackage{mathtools, amssymb, amsfonts, dsfont, nccmath}
\usepackage{microtype}
\usepackage{subcaption}
\usepackage{algorithmic}
\usepackage{graphicx}
\usepackage{textcomp}
\usepackage{xcolor}
\usepackage{amsthm} 
\usepackage{xcolor}

\newtheorem{lemma}{Lemma}
\newtheorem{theorem}{Theorem}
\theoremstyle{definition}
\newtheorem{assumption}{Assumption}
\newtheorem{definition}{Definition}
\newtheorem{remark}{Remark}
\newtheorem{proposition}{Proposition}
\newtheorem{corollary}{Corollary}
\graphicspath{{pics/}}
\begin{document}
\title{Mean Field Games on Large Expander Graphs}
\author{
    Tao Zhang,
    Peter E. Caines
\thanks{T. Zhang and P. E. Caines are with the Department of Electrical and Computer Engineering, McGill University, Montr\'eal, QC, H3A 0E9, Canada (e-mail: tao.zhang6@mail.mcgill.ca; peter.caines@mcgill.ca). This work is in part supported by AFOSR Grant FA9550-23-1-0015 and NSERC Grant RGPIN-2019-05336. 
A preliminary and condensed version of this work has been accepted for presentation at the 65th IEEE Conference on Decision and Control (CDC 2026). This preprint provides the complete proofs and technical details.
}
}

\maketitle

\begin{abstract}
This paper investigates mean field games on sparse networks. In the case of large expander graphs, the limit topologies are analyzed using the graphexon framework, which characterizes sparse connections. We prove that the associated sequence of discrete averaging operators converges strongly to a continuous operator and this is illustrated in the development of infinite limits of Gabber-Galil-Margulis expander graphs \cite{gabber1981explicit}. These properties enable the formulation and existence proof of equilibria for linear-quadratic mean field games in which each agent is identified by a spatial network label $\alpha \in X$ and only interacts with the neighborhood average characterized by the operator $\mathcal{G}$, i.e., the average state of a limited number of connected neighbors in a large expander graph. Furthermore, algebraic conditions induced by the spectral gap of $\mathcal{G}$ for the global asymptotic stability of the closed-loop system are established, and parameter thresholds that give rise to a Turing-type topological instability are established.
\end{abstract}

\section{Introduction}

Mean Field Game (MFG) theory, introduced by Lasry and Lions \cite{lasry2006jeux, lasry2007mean} and Huang, Malhamé, and Caines \cite{huang2006large,huang2010large}, provides a method to analyze large decision systems of interacting agents. Traditional MFG theory assumes uniform, all-to-all interactions. To model heterogeneous network topologies, the graphon MFG framework was developed, replacing the uniform mean-field coupling with a limit object defined by a graphon \cite{caines2021graphon, gao2020linear}. A graphon represents the limit of a sequence of dense graphs where the number of edges scales with the population size $N$ \cite{lovasz2012large}.
Let $A_N=(a_{ij}^N)_{i,j=1}^N$ be the adjacency matrix of a finite graph, let $d_i^N=\sum_{j=1}^N a_{ij}^N$, and define the degree-normalized network interaction by $[O_Nf]_i=\frac{1}{d_i^N}\sum_{j=1}^N a_{ij}^Nf_j$. In the dense graphon regime, $d_i^N=\Theta(N)$ and $O_Nf$ converges to a normalized graphon integral operator. For a bounded-degree sparse graph sequence, $d_i^N=O(1)$, and the finite-neighbor average instead converges to a finite-sum graphexon operator \cite{caines2022embedded, caines2024sparse, caines2025cdc_mean} induced by measure-preserving transformations.

To analyze limits of sparse graphs, a range of methods have been developed. For example, Borgs et al. developed the theory of $L^p$ graphons and graphings \cite{borgs2018sparse}, focusing on bounded-degree graph sequences \cite{veitch2015class}. In parallel, the concept of local weak convergence (or Benjamini-Schramm convergence) provides a probabilistic framework for analyzing the local neighborhood structure of sparse graph sequences and their application to stochastic games \cite{benjamini2001recurrence, lacker2023local}. To provide a unified methodology handling both dense and sparse network limits, the graphexon and the Laplexion framework was introduced \cite{caines2022embedded, caines2024sparse, caines2025cdc_mean,zhang2026mean}. A graphexon encodes the asymptotic edge distribution as a symmetric Borel measure on a product space.

We first construct the finite graph topology using a discrete approximation of a group action, utilizing Gabber-Galil-Margulis expander graphs \cite{gabber1981explicit, hoory2006expander,lubotzky2012expander}. Expander graphs are connected sparse graphs that maintain bounded degrees while exhibiting specific algebraic mixing properties \cite{hoory2006expander}. 
Such sparse algebraic topologies have also motivated interconnection networks for multiprocessors \cite{akers2002group} and expander-based data-center architectures such as Xpander \cite{valadarsky2016xpander}.

The primary contribution of this work is extending Lubotzky's algebraic framework for finite expander graphs to the continuous graphexon limit. Building upon this framework, the paper extends the study of Mean Field Games (MFGs) on sparse networks in two directions. First, the weak convergence of empirical measures for the underlying finite graph structure to a limit graphexon measure is established, and the strong $L^2$-convergence of the associated discrete averaging operators to an ergodic continuous integration operator is proved. Second, the infinite-horizon discounted scalar linear-quadratic-Gaussian (LQG) MFG on this limit space is formulated. The existence of a stabilizing decoupling operator and the exact algebraic boundaries for the stability is established, characterizing the Turing-type topological instability phenomenon.

The paper is organized as follows. Section \ref{sec:finite_graph} introduces finite Schreier expander graphs and discrete graphexons. Section \ref{sec:weak_convergence} addresses continuous graphexon limits and the associated operator convergence. Section \ref{sec:margulis} presents the specific construction of Gabber-Galil-Margulis expander graphs as an illustration of the continuous graphexon limit.   Section \ref{sec:mfg} formulates the infinite-horizon LQG graphexon mean field game and establishes the algebraic conditions for mean field stability and topological instability. Section \ref{sec:simulation} presents numerical simulations, followed by concluding remarks in Section \ref{sec:conclusion}.

\section{Finite Schreier Expander Graphs and Discrete Graphexons}\label{sec:finite_graph}

In this work, the connectivity of the finite graphs is generated by a discrete approximation of a group action, constructing a sequence of Schreier graphs \cite{hoory2006expander}.
\subsection{Group Theory and Property $\tau$}
\label{sec:group_theory}

Given a constant integer $K_\Lambda \ge 2$, let $\Lambda$ be an infinite discrete group generated by the finite set $S_\Lambda \triangleq \{s_1, \dots, s_{K_\Lambda}\} \subset \Lambda$. Define the inverse set of $S_\Lambda$ by $S_\Lambda^{-1} \triangleq \{s^{-1} \mid s \in S_\Lambda\}$. In this paper, it is assumed that the group identity element $e \notin S_\Lambda$ and $S_\Lambda$ is symmetric in the sense that $S_\Lambda = S_\Lambda^{-1}$.

Let $N \in \mathbb{N}$ denote the resolution parameter and $\{\Lambda_N\}_{N=1}^\infty$ be a sequence of subgroups of $\Lambda$. For each $N$, define the left coset space $V_N \triangleq \Lambda / \Lambda_N$. The index of $\Lambda_N$ in $\Lambda$ is defined as the cardinality of the coset space, denoted by $[\Lambda : \Lambda_N] \triangleq |V_N|$. It is assumed that each $\Lambda_N$ is a finite-index subgroup, satisfying $M_N \triangleq |V_N| < \infty$. A specific example can be found in Section \ref{sec:margulis}.

We assume $M_N \to \infty$ as $N \to \infty$. Let $\ell^2(V_N)$ be the space of real-valued functions on the index set equipped with the normalized inner product $\langle \phi, \psi \rangle_{\ell^2} = \frac{1}{M_N} \sum_{v \in V_N} \phi(v)\psi(v)$, and $\ell^2_0(V_N)$ be the subspace of functions orthogonal to the constant functions.

Denote a sequence of $K_\Lambda$ bijections on $V_N$ by $S_N = (\sigma_{N, 1}, \dots, \sigma_{N, K_\Lambda})$
where each $\sigma_{N, k}$ is defined by the left action of the generators in $S_\Lambda$ on the coset space $V_N$
$\sigma_{N,k}(g\Lambda_N) \triangleq (s_k g)\Lambda_N$
for any coset $g\Lambda_N \in V_N$ with $g \in \Lambda$.

Let $\pi_N$ denote the unitary representation of $\Lambda$ on the Hilbert space $\ell^2_0(V_N)$ defined by the left action $ [\pi_N(g) f](v) \triangleq f(g^{-1} v)$ for any $g \in \Lambda$ and $f \in \ell^2_0(V_N)$.

\begin{definition}[Property $(\tau)$ \cite{lubotzky1994discrete}]\label{def:property_tau}
The group $\Lambda$ has Property $(\tau)$ with respect to $\{\Lambda_N\}$ if there exists a uniform constant $\kappa > 0$ such that for any $N \ge 1$ and any $f \in \ell^2_0(V_N)$, the following inequality holds 
\[
 \max_{s \in S_\Lambda} \|\pi_N(s) f - f\|_{\ell^2} \ge \kappa \|f\|_{\ell^2},
\]
where $\kappa$ is termed a \textit{Kazhdan constant}.
\end{definition}

\subsection{Construction of Expander Graphs}

\begin{definition}[Expander Graph Sequence \cite{hoory2006expander}]
Let $\{G_N = (V_N, E_N)\}_{N=1}^\infty$ be a sequence of finite $d_N$-regular undirected multigraphs with adjacency matrix $A_N$. Define the normalized adjacency matrix by $O_N\triangleq d_N^{-1}A_N$, and let $\lambda_2(O_N)$ denote its second largest eigenvalue. The sequence $\{G_N\}_{N=1}^\infty$ constitutes a family of \textit{expander graphs} if there exists a uniform constant $c_{gap} > 0$ such that 
\[
1 - \lambda_2(O_N) \ge c_{gap}
\]
for any $N \ge 1$.
\end{definition}
To construct an expander graph sequence, we define an undirected $K_\Lambda$-regular multigraph $G_N = (V_N, E_N)$ where the finite coset space $V_N=\Lambda/\Lambda_N$ is identified as the vertex set and the edges are defined by
\[
    E_N \triangleq \left\{ \{v, \sigma_{N, k}(v)\} \mid v \in V_N, \, k \in \{1, \dots, K_\Lambda\} \right\}.
\]
The edges in this expression are counted with multiplicity over the index pairs $(v,k)$.

Let $O_N: \ell^2(V_N) \to \ell^2(V_N)$ denote the normalized discrete adjacency operator, evaluated pointwise by
\[
    [O_N f](v) \triangleq \frac{1}{K_\Lambda} \sum_{k=1}^{K_\Lambda} f(\sigma_{N, k}(v)).
\]
\begin{remark}
    $G_N$ is the finite Schreier graph of the action of $\Lambda$ on $\Lambda/\Lambda_N$ generated by $S_\Lambda$ and reduces to the corresponding Cayley graph when $\Lambda_N$ is trivial \cite{hoory2006expander,lubotzky2012expander}.
\end{remark}
\begin{figure}[t]
    \centering
    \begin{subfigure}[b]{0.38\linewidth}
        \centering
        \includegraphics[width=\linewidth]{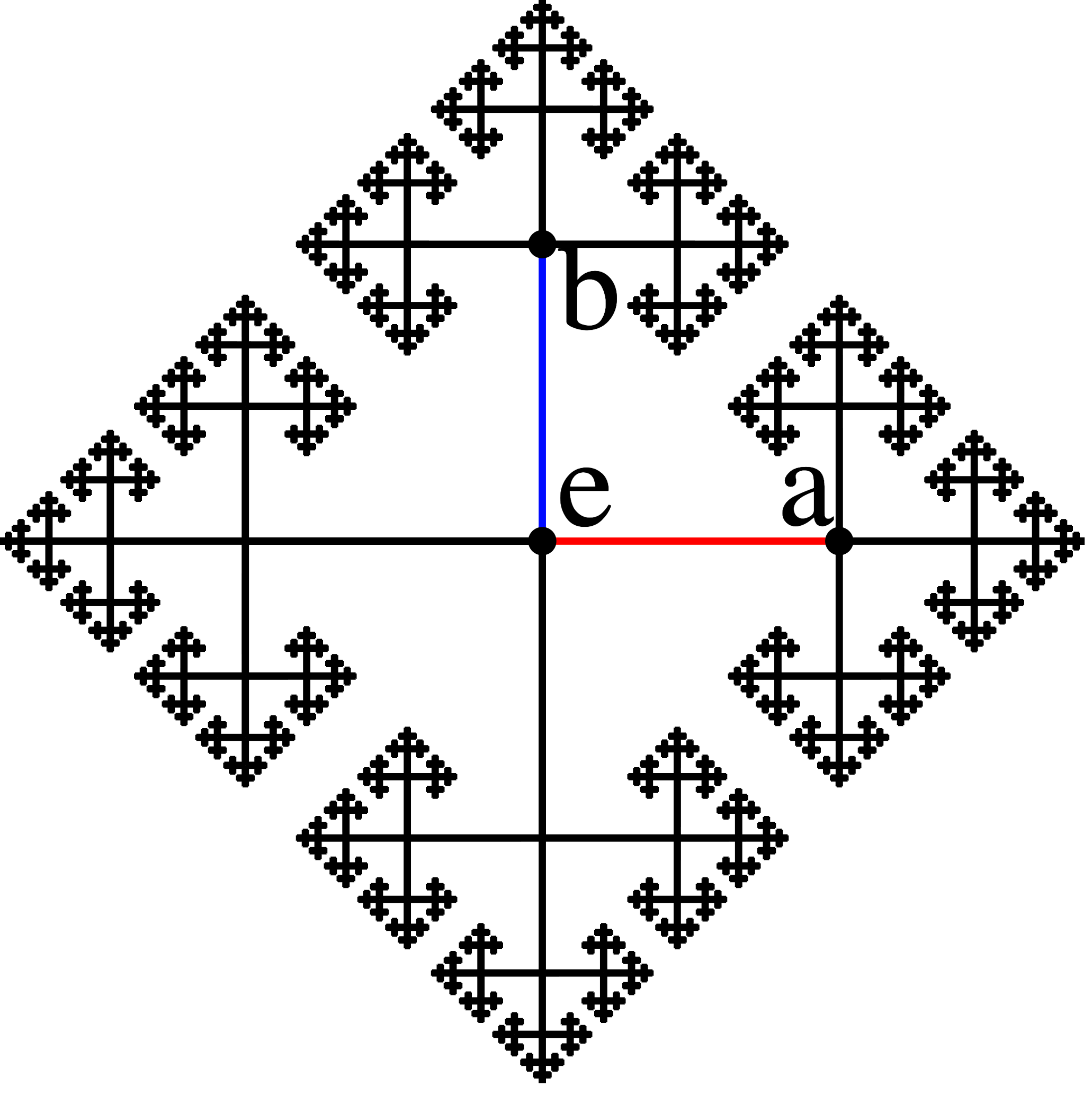}
        \caption{Cayley graph}
    \end{subfigure}
    \hfill
    \begin{subfigure}[b]{0.52\linewidth}
        \centering
        \includegraphics[width=\linewidth]{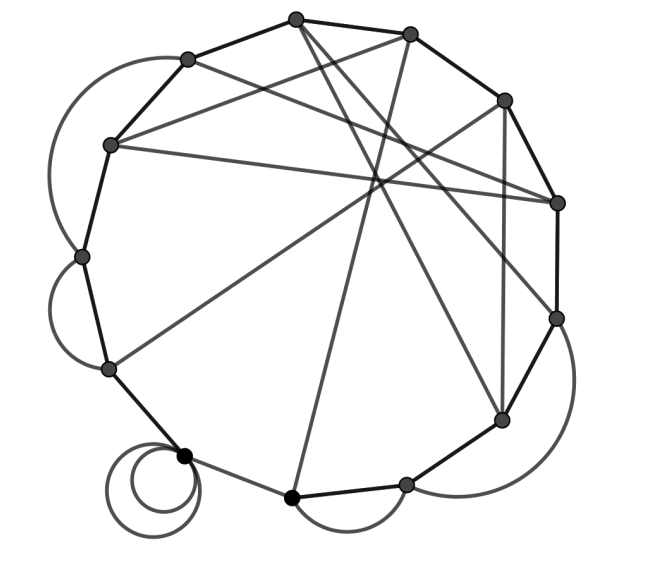}
        \caption{Schreier graph}
    \end{subfigure}
    \caption{A Cayley graph with two generators (image source: Wikipedia) and a Schreier graph of an affine group acting on a finite field with 13 elements \cite{sabatini2025groups}.}
    \label{fig:algebraic_graphs}
\end{figure}

\begin{assumption}\label{assm:two_sided_expander}
    \begin{itemize}
        \item The group $\Lambda$ has Property $(\tau)$ with respect to $\{\Lambda_N\}$.
        \item To exclude configurations that asymptotically approach a regular bipartite structure, we assume that the minimal eigenvalue of $O_N$ satisfies $\lambda_{\min}(O_N) \ge -1 + c_{bip}$ for a uniform constant $c_{bip} > 0$. 
    \end{itemize}
\end{assumption}
\begin{remark}
   Property $(\tau)$ provides a uniform gap below the principal eigenvalue $1$, while the constant $c_{bip}$ bounds the lower spectral edge away from $-1$. Together, they provide the absolute spectral gap inherited by the limit operator $\mathcal{G}$ and used to evaluate the Turing-type instability thresholds in Section \ref{sec:mfg}.
\end{remark}

\begin{lemma}\label{lem:schreier_expander}
Under Assumption \ref{assm:two_sided_expander}, the sequence $\{G_N\}_{N=1}^\infty$ constitutes a family of expander graphs with $c_{gap}={\kappa^2}/{(2K_\Lambda)}$. Moreover, there exists $c_{abs}>0$ such that $\|O_N\|_{\mathcal{L}(\ell^2_0)}=\max\{\lambda_2(O_N),|\lambda_{\min}(O_N)|\}\leq1-c_{abs}$ for every $N\geq1$.
\end{lemma}

\begin{proof}
For any $f\in\ell^2_0(V_N)$, the symmetry of $S_\Lambda$ yields
\[
\langle (I-O_N)f,f\rangle_{\ell^2}=\frac{1}{2K_\Lambda}\sum_{s\in S_\Lambda}\|\pi_N(s)f-f\|_{\ell^2}^2\geq\frac{\kappa^2}{2K_\Lambda}\|f\|_{\ell^2}^2.
\]
The second part of Assumption \ref{assm:two_sided_expander} then yields the absolute spectral bound with $c_{abs}=\min\{\kappa^2/(2K_\Lambda),c_{bip}\}$.
\end{proof}

\section{Continuous Graphexon Limits and Operator Convergence}\label{sec:weak_convergence}

To analyze the behavior of the discrete expander graphs, we embed the finite vertex sets $V_N$ into a continuous space. Let $(X, d)$ be a compact metric space equipped with its Borel $\sigma$-algebra $\mathcal{B}$ and a Borel probability measure $\nu$, forming the probability space $(X, \mathcal{B}, \nu)$. Let $\|\cdot\|_{L^2}$ be the induced $L^2$-norm. Denote by $L^2_0(X, \nu)$ the orthogonal complement of the constant functions in $L^2(X, \nu)$. Let $\|\cdot\|_{\mathcal{L}(L^2_0)}$ denote the induced operator norm restricted to $L^2_0(X, \nu)$. 

Let $\mathcal{P}_N = \{Q_v^N  \subset X\}_{v \in V_N}$ be a measurable partition of the space $X$ such that the measure of each cell is uniform, satisfying $ \nu(Q_v^N) = 1/M_N$ for any $ v \in V_N$.

We assume that the maximum diameter of the partition cells, denoted by $\delta_N$, vanishes as $N \to \infty$ satisfying
$ \delta_N \triangleq \max_{v \in V_N} \sup_{x, y \in Q_v^N} d(x, y)$ and $ \lim_{N \to \infty} \delta_N = 0$.

Let $\iota_N: V_N \to X$ be an embedding map that assigns a representative point to each index, satisfying $\iota_N(v) \in Q_v^N$. 
For any $x \in X$, let $i_N:X \to V_N$ denote the unique index such that $x \in Q_{i_N(x)}^N$.

To guarantee the existence of a limit for the sequence of discrete graph interactions, we impose a uniform regularity condition on the sequences.

\begin{assumption}\label{assm:continuum_limit}
There exists a uniform constant $L > 0$ such that for all $N \ge 1$, $k \in \{1, \dots, K_\Lambda\}$, and $u, v \in V_N$, the discrete mappings satisfy the Lipschitz condition on the embedded space
\[
    d(\iota_N(\sigma_{N,k}(u)), \iota_N(\sigma_{N,k}(v))) \le L d(\iota_N(u), \iota_N(v)).
\]
\end{assumption}

\begin{lemma}\label{lem:continuum_limit}
Under Assumption \ref{assm:continuum_limit}, there exists a sequence of bi-Lipschitz bijections $\sigma_{\infty,1},\dots,\sigma_{\infty,K_\Lambda}$ on $X$ with uniform Lipschitz constant $L$, and an increasing subsequence $\{N_j\}_{j=1}^\infty$, such that for each $k \in \{1, \dots, K_\Lambda\}$, 
\[
    \lim_{j \to \infty} \sup_{x \in X} d(\iota_{N_j}(\sigma_{N_j, k}(i_{N_j}(x))), \sigma_{\infty, k}(x)) = 0.
\]
\end{lemma}

\begin{proof}
Let $\mathcal{D}$ be a countable dense subset of $X$ and set $f_{N,k}(x)\triangleq\iota_N(\sigma_{N,k}(i_N(x)))$. Compactness of $X$, Cantor's diagonal argument, and finiteness of $K_\Lambda$ give a common subsequence $\{N_j\}$ on which $f_{N_j,k}(x)$ converges for every $x\in\mathcal{D}$ and every $k$. Denote the limit on $\mathcal{D}$ by $\sigma_{\infty,k}$.

Since $d(\iota_N(i_N(x)),x)\leq\delta_N$, Assumption \ref{assm:continuum_limit} gives, for $x,y\in X$,
\[
d(f_{N,k}(x),f_{N,k}(y))\leq L\bigl(d(x,y)+2\delta_N\bigr).
\]
Passing to the limit on $\mathcal{D}$ shows that $\sigma_{\infty,k}$ extends uniquely to an $L$-Lipschitz map on $X$. For a finite $\varepsilon$-net $\mathcal{D}_\varepsilon\subset\mathcal{D}$ and $y\in\mathcal{D}_\varepsilon$ with $d(x,y)<\varepsilon$, the triangle inequality yields
\[
d(f_{N_j,k}(x),\sigma_{\infty,k}(x))\leq 2L\varepsilon+2L\delta_{N_j}+d(f_{N_j,k}(y),\sigma_{\infty,k}(y)).
\]
The last term tends to zero uniformly on the finite net. Letting first $j\to\infty$ and then $\varepsilon\to0$ proves uniform convergence on $X$.

For each $k$, the inverse of $\sigma_{N,k}$ is another map in the symmetric generator sequence. Applying the same argument to these inverse maps along the common subsequence produces an $L$-Lipschitz limit $\tau_{\infty,k}$. At representative points, the two compositions are the identity. Their uniform convergence and the density of the representative points imply $\tau_{\infty,k}\circ\sigma_{\infty,k}=\sigma_{\infty,k}\circ\tau_{\infty,k}=\mathrm{id}_X$. Hence $\sigma_{\infty,k}$ is a bi-Lipschitz bijection.

The bijections $\sigma_{N_j,k}$ permute the representative points. For every $h\in C(X)$, uniform convergence and the cell approximation therefore give
\[
\int_X h\circ\sigma_{\infty,k}\,d\nu=\lim_{j\to\infty}\frac{1}{M_{N_j}}\sum_{v\in V_{N_j}}h(\iota_{N_j}(v))=\int_X h\,d\nu.
\]
Thus each limit map preserves $\nu$, as used below.
\end{proof}

For all $A \in \mathcal{B}$, the limit mappings satisfy the measure-preserving property $\nu(\sigma_{\infty,k}^{-1}(A)) = \nu(A)$. Since the original generator set $S_\Lambda$ is symmetric, the set of limit mappings remains closed under inversion.

The discretization error associating the continuous mapping with its discrete approximation is quantified by
\[
    \epsilon_{N_j, k} \triangleq \max_{v \in V_{N_j}} d(\sigma_{\infty, k}(\iota_{N_j}(v)), \iota_{N_j}(\sigma_{N_j, k}(v))).
\]
By Lemma \ref{lem:continuum_limit}, the maximum error vanishes asymptotically
\[
    \lim_{j \to \infty} \max_{k \in \{1, \dots, K_\Lambda\}} \epsilon_{N_j, k} = 0.
\]
\begin{definition}
    Let $\Sigma \triangleq \{\sigma_1, \dots, \sigma_K\}$ denote the set of distinct limit mappings derived from the sequence $(\sigma_{\infty, 1}, \dots, \sigma_{\infty, K_\Lambda})$. Let $K  \triangleq |\Sigma|\le K_\Lambda$ denote its cardinality. The limit group $\Gamma$ acting on $X$ is defined as the group generated by the symmetric set $\Sigma$ given by $\Gamma \triangleq \langle \Sigma \rangle$.
\end{definition}

Let $c_m \triangleq \sum_{k=1}^{K_\Lambda} \mathbb{I}_{\{\sigma_{\infty, k} = \sigma_m\}}$ denote the multiplicity of each distinct mapping $\sigma_m \in \Sigma$ within the unreduced limit sequence.

\begin{assumption}\label{assm:uniform_multiplicity}
For simplicity of exposition, we assume uniform multiplicity, i.e., $c_m = K_\Lambda / K$ for all $m \in \{1, \dots, K\}$.
\end{assumption}

For the non-uniform case, one could assign the weight $p_m=c_m/K_\Lambda$ to each distinct generator $\sigma_m\in\Sigma$ in \eqref{eq:W_limit}.

Define the edge space as $X \times X$ equipped with the metric $D((x_1, y_1), (x_2, y_2)) \triangleq d(x_1, x_2) + d(y_1, y_2)$. 

\begin{definition}
The \textit{empirical graphexon measure} $\mathbf{W}_N$ on the product space $X \times X$ is defined by
\begin{equation}
    \mathbf{W}_N \triangleq \frac{1}{K_\Lambda M_N} \sum_{v \in V_N} \sum_{k=1}^{K_\Lambda} \delta_{(\iota_N(v), \iota_N(\sigma_{N, k}(v)))}
\end{equation}
where $\delta_{(x, y)}$ denotes the Dirac measure at the point $(x, y) \in X \times X$.
\end{definition}

Since the underlying space $X$ is compact, the product space $(X \times X, D)$ is also compact. By Prokhorov's theorem, the uniformly tight sequence of probability measures $\{\mathbf{W}_{N_j}\}_{j=1}^\infty$ is relatively compact in the topology of weak convergence, guaranteeing the existence of weakly convergent subsequential limits.

\begin{theorem}\label{thm:measure_convergence}
Under Assumptions \ref{assm:continuum_limit} and \ref{assm:uniform_multiplicity}, the subsequence of empirical graphexon measures $\mathbf{W}_{N_j}$ converges weakly to the graphexon measure $\mathbf{W}$ given by
\begin{equation}\label{eq:W_limit}
     \mathbf{W}(dx, dy) \triangleq \frac{1}{K} \sum_{m=1}^{K} \nu(dx) \delta_{\sigma_m(x)}(dy).
\end{equation}
\end{theorem}

\begin{proof}
Let $\Psi\in C(X\times X)$ and let $\omega_\Psi$ be its modulus of continuity for the metric $D$. All sums over $v$ below range over $V_{N_j}$. By definition,
\[
\int_{X^2}\Psi\,d\mathbf{W}_{N_j}=\frac{1}{K_\Lambda M_{N_j}}\sum_{k=1}^{K_\Lambda}\sum_v\Psi\bigl(\iota_{N_j}(v),\iota_{N_j}(\sigma_{N_j,k}(v))\bigr).
\]
Assumption \ref{assm:uniform_multiplicity} rewrites the limiting integral using the unreduced sequence of maps. Since each cell has measure $1/M_{N_j}$,
\[
\int_{X^2}\Psi\,d\mathbf{W}=\frac{1}{K_\Lambda}\sum_{k=1}^{K_\Lambda}\sum_v\int_{Q_v^{N_j}}\Psi\bigl(x,\sigma_{\infty,k}(x)\bigr)\,\nu(dx).
\]
For $x\in Q_v^{N_j}$, put $u=\sigma_{N_j,k}(v)$. The cell-diameter bound and the definition of $\epsilon_{N_j,k}$ give
\[
D\bigl((x,\sigma_{\infty,k}(x)),(\iota_{N_j}(v),\iota_{N_j}(u))\bigr)\leq(1+L)\delta_{N_j}+\epsilon_{N_j,k}.
\]
With $\epsilon_{N_j}\triangleq\max_k\epsilon_{N_j,k}$, integration over the cells therefore yields
\[
\left|\int_{X^2}\Psi\,d\mathbf{W}_{N_j}-\int_{X^2}\Psi\,d\mathbf{W}\right|\leq\omega_\Psi\bigl((1+L)\delta_{N_j}+\epsilon_{N_j}\bigr).
\]
Both errors tend to zero, so the integrals converge for every continuous test function. This is the asserted weak convergence.
\end{proof}
\begin{remark}
Within the graphexon framework \cite{caines2022embedded,caines2024sparse}, any graphexon measure $\mathbf{G}$ on the product space $X \times X$ admits a Lebesgue decomposition $ \mathbf{G}(dx, dy) = g(x, y) \nu(dx) \nu(dy) + \mathbf{N}(dx, dy)$
where $g \in L^1(X \times X)$ represents the absolutely continuous component capturing dense connections, and $\mathbf{N}$ is the singular measure component capturing sparse connections. 
In this paper, the absolutely continuous component is $g(x, y) = 0$ since the measure is entirely supported on the graphs of the measure-preserving bijections $\sigma_m$. Consequently, the edge topology is completely characterized by the singular component $\mathbf{N} = \mathbf{W}$.
\end{remark}

\begin{definition}
The projection operator $P_N: L^2(X, \nu) \to \ell^2(V_N)$ is defined by the local cell average
\[
    [P_N f](v) \triangleq \frac{1}{\nu(Q_v^N)} \int_{Q_v^N} f(z) \nu(dz).
\]
The extension operator $E_N: \ell^2(V_N) \to L^2(X, \nu)$ constructs a piecewise constant function from discrete values
\[
    [E_N \phi](x) \triangleq \sum_{v \in V_N} \phi(v) \mathbb{I}_{Q_v^N}(x).
\]
\end{definition}

\begin{definition}[Step-Graphexon Operator]
The step-graphexon operator $\mathcal{G}_N: L^2(X, \nu) \to L^2(X, \nu)$ is the composition $\mathcal{G}_N \triangleq E_N \circ O_N \circ P_N$. For any $x \in X$, 
\begin{equation} \label{eq:step_operator}
    [\mathcal{G}_N f](x) = \frac{1}{K_\Lambda} \sum_{k=1}^{K_\Lambda} [P_N f](\sigma_{N, k}(i_N(x))).
\end{equation}
\end{definition}
\begin{definition}[Graphexon Operator]
The \textit{graphexon operator} $\mathcal{G}: L^2(X, \nu) \to L^2(X, \nu)$ is defined by $[\mathcal{G} f](x) \triangleq \int_X f(y) \mathbf{W}(x, dy)$
for any $f \in L^2(X, \nu)$.
\end{definition}

Computing the integral yields the form
\begin{equation}
    [\mathcal{G} f](x) = \frac{1}{K} \sum_{m=1}^{K} f(\sigma_m(x)).
\end{equation}

\begin{remark}\label{rem:physical_motivation}
The operator $\mathcal{G}$ establishes the coupling mechanism for the multi-agent mean field system. While standard graphon frameworks model dense limit networks via integral operators (e.g., $\int_X W(\alpha, \beta) m(\beta) \nu(d\beta)$) representing infinite all-to-all weighted couplings, the graphexon operator captures bounded-degree sparse limits. It replaces the dense integration with finite discrete summations over measure-preserving spatial shifts.
\end{remark}

\begin{proposition}\label{prop:absolute_spectral_gap}
Under Assumptions \ref{assm:two_sided_expander}, \ref{assm:continuum_limit}, and \ref{assm:uniform_multiplicity}, the spectral radius of the limit graphexon operator $\mathcal{G}$ restricted to $L^2_0(X, \nu)$ satisfies $ \|\mathcal{G}\|_{\mathcal{L}(L^2_0)} \le 1-c_{abs} < 1$.
\end{proposition}

\begin{proof}
The equal cell measures make $E_{N_j}$ an isometry from $\ell^2(V_{N_j})$ to $L^2(X,\nu)$. Its adjoint $P_{N_j}$ is a contraction. Both maps preserve the zero-mean subspaces, so Assumption \ref{assm:two_sided_expander} gives
\[
\|\mathcal{G}_{N_j}\|_{\mathcal{L}(L^2_0)}\leq\|O_{N_j}\|_{\mathcal{L}(\ell^2_0)}\leq1-c_{abs}.
\]
By Proposition \ref{prop:strong_convergence}, $\mathcal{G}_{N_j}f\to\mathcal{G}f$ in $L^2$ for every $f\in L^2_0(X,\nu)$. Hence
\[
\|\mathcal{G}f\|_{L^2}=\lim_{j\to\infty}\|\mathcal{G}_{N_j}f\|_{L^2}\leq(1-c_{abs})\|f\|_{L^2}.
\]
Taking the supremum over unit vectors $f\in L^2_0(X,\nu)$ proves the claim. By Lemma \ref{lem:spectral_properties}, $\mathcal{G}$ is self-adjoint, so its spectral radius on this subspace equals its operator norm.
\end{proof}

\begin{lemma}\label{lem:spectral_properties}
The limit operator $\mathcal{G}$ is self-adjoint on $L^2(X, \nu)$ and its spectrum $\sigma(\mathcal{G})$ is real.
\end{lemma}

\begin{proof}
For any $f, g \in L^2(X, \nu)$, the inner product is
\[
    \langle \mathcal{G}f, g \rangle = \frac{1}{K} \sum_{k=1}^K \int_X f(\sigma_{k}(x)) g(x) \nu(dx).
\]
Applying the measure-preserving change of variables $y = \sigma_{k}(x)$ provides
\[
    \langle \mathcal{G}f, g \rangle = \int_X f(y) \left( \frac{1}{K} \sum_{k=1}^K g(\sigma_{k}^{-1}(y)) \right) \nu(dy).
\]
Since $\Sigma$ is closed under inversion up to a permutation of its indices, the summation over the inverse sequence is a rearrangement of the forward summation, which is $[\mathcal{G}g](y)$. Thus $\langle \mathcal{G}f, g \rangle = \langle f, \mathcal{G}g \rangle$, confirming that $\mathcal{G}$ is self-adjoint and $\sigma(\mathcal{G}) \subset \mathbb{R}$.
\end{proof}

\begin{proposition}\label{prop:strong_convergence}
Under Assumptions \ref{assm:continuum_limit} and \ref{assm:uniform_multiplicity}, for any $f \in L^2(X, \nu)$, the subsequence of operators $\{\mathcal{G}_{N_j}\}_{j=1}^\infty$ converges strongly to $\mathcal{G}$ in the $L^2$-norm, satisfying $ \lim_{j \to \infty} \| \mathcal{G}_{N_j} f - \mathcal{G} f \|_{L^2} = 0$.
\end{proposition}

\begin{proof}
First let $f:X\to\mathbb{R}$ be Lipschitz with constant $L_f$. For $x\in Q_v^{N_j}$ and $u=\sigma_{N_j,k}(v)$, the definitions of $P_{N_j}$ and $\epsilon_{N_j,k}$ give
\[
\left|[P_{N_j}f](u)-f(\sigma_{\infty,k}(x))\right|\leq L_f\bigl((1+L)\delta_{N_j}+\epsilon_{N_j,k}\bigr).
\]
Indeed, for $z\in Q_u^{N_j}$, insert the representative points $\iota_{N_j}(u)$ and $\iota_{N_j}(v)$ between $z$ and $\sigma_{\infty,k}(x)$. The resulting three distances are bounded by $\delta_{N_j}$, $\epsilon_{N_j,k}$, and $L\delta_{N_j}$, respectively.

Set $\epsilon_{N_j}=\max_k\epsilon_{N_j,k}$. The uniform multiplicity assumption identifies the average over the $K_\Lambda$ limit maps with $\mathcal{G}f$. Thus, by \eqref{eq:step_operator},
\[
\|\mathcal{G}_{N_j}f-\mathcal{G}f\|_{L^2}\leq L_f\bigl((1+L)\delta_{N_j}+\epsilon_{N_j}\bigr)\longrightarrow0.
\]
Lipschitz functions are dense in $L^2(X,\nu)$. Moreover, $P_{N_j}$, $O_{N_j}$, and $E_{N_j}$ are contractions, and $\|\mathcal{G}\|\leq1$ because the limit maps preserve $\nu$. Approximating any $f\in L^2(X,\nu)$ by Lipschitz functions and using these uniform bounds extends the convergence to all of $L^2(X,\nu)$.
\end{proof}

\begin{corollary}
Under Assumptions \ref{assm:two_sided_expander}, \ref{assm:continuum_limit}, and \ref{assm:uniform_multiplicity}, the principal eigenvalue $\lambda = 1$ of the graphexon operator $\mathcal{G}$ is simple, and its associated eigenspace is $\text{span}\{\mathbf{1}\}$.
\end{corollary}

\begin{proof}
Suppose $f \in L^2(X, \nu)$ satisfies $\mathcal{G}f = f$. The function admits an orthogonal decomposition $f = c\mathbf{1} + f_0$, where $c \in \mathbb{R}$ and $f_0 \in L^2_0(X, \nu)$. Applying the graphexon operator yields $\mathcal{G}f = c\mathcal{G}\mathbf{1} + \mathcal{G}f_0 = c\mathbf{1} + \mathcal{G}f_0$. It follows that $\mathcal{G}f_0 = f_0$. By Proposition \ref{prop:absolute_spectral_gap}, taking the $L^2$-norm gives
\[
    \|f_0\|_{L^2} = \|\mathcal{G}f_0\|_{L^2} \le \|\mathcal{G}\|_{\mathcal{L}(L^2_0)} \|f_0\|_{L^2} \le (1-c_{abs}) \|f_0\|_{L^2}.
\]
Since $c_{abs} > 0$, this inequality is satisfied if and only if $\|f_0\|_{L^2} = 0$, implying $f_0 = 0$ almost everywhere. Therefore, $f = c\mathbf{1}$, establishing that the eigenspace corresponding to $\lambda = 1$ is $\text{span}\{\mathbf{1}\}$.
\end{proof}

\section{Graph Construction on the Flat Torus}\label{sec:margulis}

Consider an example on the two-dimensional torus $X = \mathbb{T}^2 \cong \mathbb{R}^2 / \mathbb{Z}^2$ equipped with the standard Euclidean metric and the Lebesgue probability measure $\nu$. Let $\Gamma$ be the linear subgroup of $\text{SL}(2, \mathbb{Z})$ generated by
\[
T_1=\begin{pmatrix}1&2\\0&1\end{pmatrix},\qquad T_2=\begin{pmatrix}1&0\\2&1\end{pmatrix}.
\]
Define the infinite discrete group $\Lambda \triangleq \mathbb{Z}^2 \rtimes \Gamma$, where $\rtimes$ denotes the semi-direct product equipped with the group operation $(z_1, \tilde{T}_1)(z_2, \tilde{T}_2) = (z_1 + \tilde{T}_1 z_2, \tilde{T}_1 \tilde{T}_2)$ for any $(z_1, \tilde{T}_1), (z_2, \tilde{T}_2) \in \Lambda$. Define the sequence of finite-index subgroups $\Lambda_N \triangleq (N\mathbb{Z})^2 \rtimes \Gamma$.

A sequence of expander graphs is obtained using the discrete algebraic construction of the Gabber-Galil-Margulis expander \cite{hoory2006expander}. Let $ \{e_1, e_2\} \subset \mathbb{Z}^2$ denote the standard basis vectors.
Define the finite symmetric generating set $S_\Lambda = \{s_1, \dots, s_8\} \subset \Lambda$ by 
\begin{align*}
    s_1 &= (0, T_1), & s_2 &= (e_1, T_1), \\
    s_3 &= (0, T_2), & s_4 &= (e_2, T_2), \\
    s_5 &= (0, T_1^{-1}), & s_6 &= (-T_1^{-1}e_1, T_1^{-1}), \\
    s_7 &= (0, T_2^{-1}), & s_8 &= (-T_2^{-1}e_2, T_2^{-1}).
\end{align*}

Let $V_N = (\mathbb{Z}/N\mathbb{Z})^2$ denote the finite discrete index set, identified with the left coset space $\Lambda / \Lambda_N$ with cardinality $M_N = N^2$. The embedding map $\iota_N: V_N \to \mathbb{T}^2$ assigns coordinates on the torus by $\iota_N(v) = \frac{v}{N} \bmod 1$.

For $v=(v_1,v_2)\in V_N$, let $Q_v^N$ be the half-open square $[v_1/N,(v_1+1)/N)\times[v_2/N,(v_2+1)/N)$ modulo $1$. Then $\{Q_v^N\}_{v\in V_N}$ is a measurable partition of $\mathbb{T}^2$ satisfying $\nu(Q_v^N)=1/N^2=1/M_N$ and $\delta_N\leq\sqrt{2}/N\to0$.

The discrete graph topology is governed by an indexed family $S_N = (\sigma_{N, 1}, \dots, \sigma_{N, 8})$ of $K_\Lambda = 8$ affine bijections acting on the discrete space $V_N$. For a vector index $v \in V_N$, define the left actions of the generators in $S_\Lambda$ as
\begin{align*}
    \sigma_{N, 1}(v) &= T_1 v \bmod N, & \sigma_{N, 2}(v) &= (T_1 v + e_1) \bmod N, \\
    \sigma_{N, 3}(v) &= T_2 v \bmod N, & \sigma_{N, 4}(v) &= (T_2 v + e_2) \bmod N,\\
    \sigma_{N, 5}(v) &= T_1^{-1} v \bmod N, & \sigma_{N, 6}(v) &= T_1^{-1}(v - e_1) \bmod N, \\
    \sigma_{N, 7}(v) &= T_2^{-1} v \bmod N, & \sigma_{N, 8}(v) &= T_2^{-1}(v - e_2) \bmod N.
\end{align*}

Assumption \ref{assm:continuum_limit} is verified directly for this construction. Write $\sigma_{N,k}(v)=A_kv+b_k\bmod N$, where $A_k\in\{T_1,T_2,T_1^{-1},T_2^{-1}\}$ and $b_k\in\mathbb{Z}^2$. For the quotient Euclidean metric on $\mathbb{T}^2$, the integrality of $A_k$ gives
\[
d(\iota_N(\sigma_{N,k}(u)),\iota_N(\sigma_{N,k}(v)))\leq\|A_k\|_2d(\iota_N(u),\iota_N(v)).
\]
Since $\|T_1\|_2=\|T_2\|_2=\|T_1^{-1}\|_2=\|T_2^{-1}\|_2=1+\sqrt{2}$, Assumption \ref{assm:continuum_limit} holds uniformly in $N$ with $L=1+\sqrt{2}$.

Let $O_N$ denote the normalized adjacency operator on $\ell^2(V_N)$ defined by the $K_\Lambda=8$ discrete bijections. The expander property of this sequence is guaranteed by bounding its absolute spectral radius on the zero-mean subspace.

\begin{lemma}[\cite{gabber1981explicit, hoory2006expander}]
\label{lem:gabber_galil_bound}
For the Gabber-Galil-Margulis expander graph sequence, there exists a spectral gap such that the adjacency operator $O_N$ restricted to $\ell^2_0(V_N)$ satisfies
$\|O_N\|_{\mathcal{L}(\ell^2_0)} = \max \{ \lambda_2(O_N), |\lambda_{\min}(O_N)| \} \le \frac{5\sqrt{2}}{8}$ for any $N \ge 1$.
\end{lemma}

Lemma \ref{lem:gabber_galil_bound} verifies Assumption \ref{assm:two_sided_expander} for this construction. In particular, one may take $c_{abs}=1-5\sqrt{2}/8$.

As the resolution parameter $N \to \infty$, the embedded shift vectors scaled by the resolution $e_j/N$ vanish. The continuous limit connections are governed by $(\sigma_{\infty, k})_{k=1}^8$. For any $x \in \mathbb{T}^2$, the limit mappings converge pointwise
\begin{align*}
    \sigma_{\infty, 1}(x) &= \sigma_{\infty, 2}(x) = T_1 x \bmod 1, \\
    \sigma_{\infty, 3}(x) &= \sigma_{\infty, 4}(x) = T_2 x \bmod 1, \\
    \sigma_{\infty, 5}(x) &= \sigma_{\infty, 6}(x) = T_1^{-1} x \bmod 1, \\
    \sigma_{\infty, 7}(x) &= \sigma_{\infty, 8}(x) = T_2^{-1} x \bmod 1.
\end{align*}

Each distinct limit mapping occurs twice in $(\sigma_{\infty,k})_{k=1}^8$. Hence $c_m=2=K_\Lambda/K$ for all $m\in\{1,\dots,4\}$, and Assumption \ref{assm:uniform_multiplicity} holds.

The limit graphexon operator $\mathcal{G}: L^2(\mathbb{T}^2, \nu) \to L^2(\mathbb{T}^2, \nu)$ generated by the minimal set $\Sigma=\{T_1, T_2, T_1^{-1}, T_2^{-1}\}$ of size $K=4$ satisfies
\[
[\mathcal{G} f](x)=\frac{1}{4}\bigl(f(T_1x)+f(T_1^{-1}x)+f(T_2x)+f(T_2^{-1}x)\bigr).
\]

Thus Assumptions \ref{assm:two_sided_expander}--\ref{assm:uniform_multiplicity} hold for this construction. Theorem \ref{thm:measure_convergence} and Propositions \ref{prop:absolute_spectral_gap} and \ref{prop:strong_convergence} therefore apply.

Before implementing the main theorem of this subsection, we present the following lemmas.
\begin{lemma}[\cite{colin2013scattering}]
\label{lem:tree_finite_support_spectrum}
Let $\mathcal{T}$ be an infinite $(q+1)$-regular tree with vertex set $V$, where $q\geq2$, and let $A_0$ be its unnormalized adjacency operator on $\ell^2(V)$. Let $W$ be a self-adjoint operator whose matrix entries $W_{xy}$ satisfy
\[
W_{xy}=0\qquad\text{whenever }x\notin K\text{ or }y\notin K
\]
for some finite set $K\subset V$. Then the spectral decomposition of $A_0+W$ is
\[
\ell^2(V)=\mathcal{H}_{ac}\oplus\mathcal{H}_{pp},
\]
where $\mathcal{H}_{ac}$ is the absolutely continuous spectral subspace and $\mathcal{H}_{pp}$ is the closed span of its eigenfunctions. In other words, $A_0+W$ has no singular continuous spectrum.
\end{lemma}

To apply Lemma~\ref{lem:tree_finite_support_spectrum}, consider a graph $G$ consisting of a finite cycle with two infinite rooted trees attached to each cycle vertex, each vertex of an attached tree having three children. Let $A_G$ be its unnormalized adjacency operator. Edges are counted with multiplicity, and a self-loop contributes twice to the adjacency operator.
\begin{figure}[t]
    \centering
    \includegraphics[width=\linewidth]{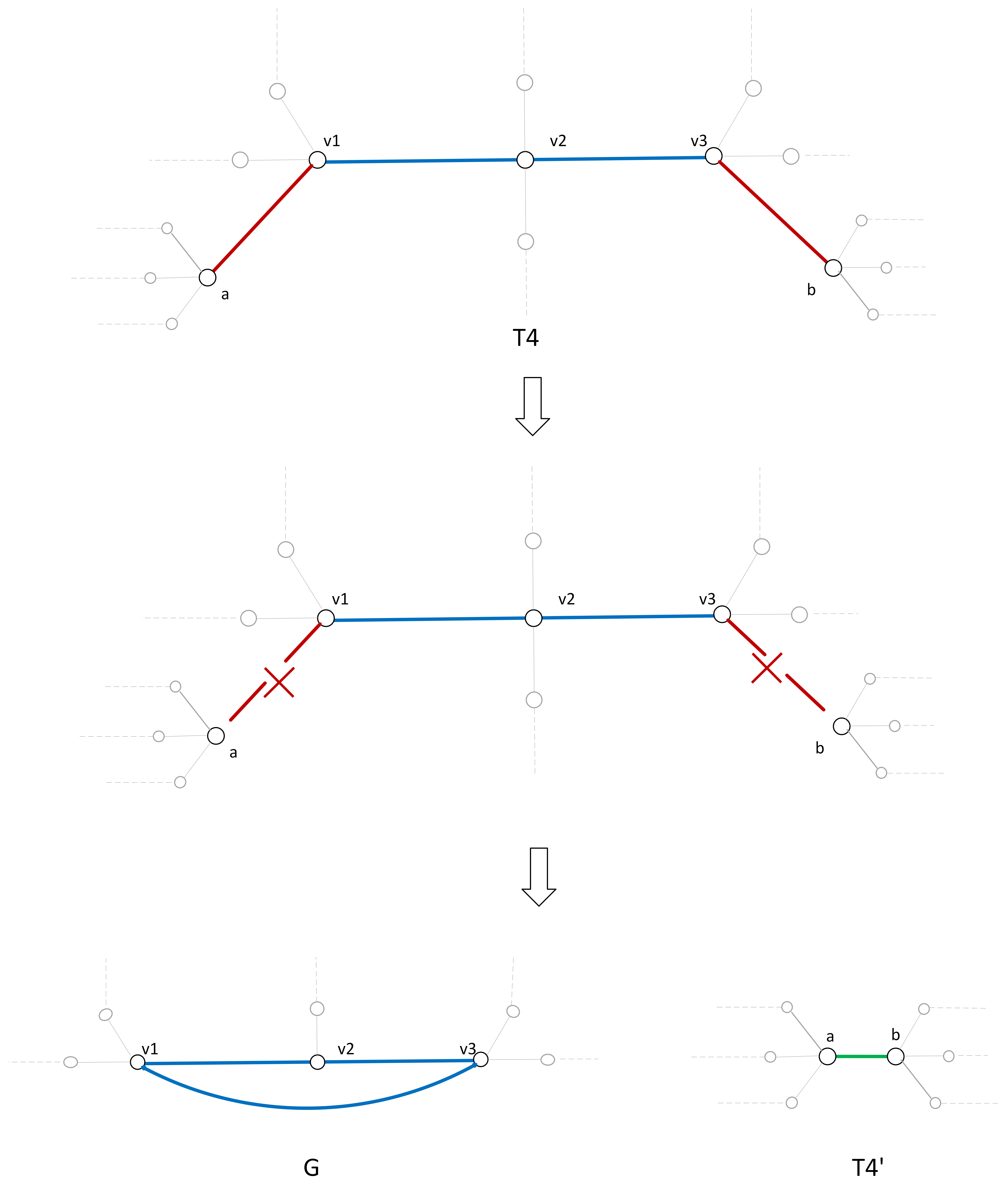}
    \caption{Finite edge surgery on $T_4$. Cutting two edges and reconnecting their endpoints produces the cycle graph $G$ and another copy of $T_4$.}
    \label{fig:tree_surgery}
\end{figure}
\begin{lemma}
\label{lem:cycle_tree_surgery}
The operator $A_G$ has no singular continuous spectrum.
\end{lemma}

\begin{proof}
We construct a finite-support perturbation to which Lemma~\ref{lem:tree_finite_support_spectrum} applies. Let $\ell$ be the number of vertices in the core cycle of $G$, and choose a simple path $v_1,\ldots,v_\ell$ in the infinite $4$-regular tree $T_4$. Choose an off-path neighbor $a$ of $v_1$ and an off-path neighbor $b$ of $v_\ell$, with $a\neq b$. Delete $\{v_1,a\}$ and $\{v_\ell,b\}$, then add an edge joining $v_1$ to $v_\ell$ and an edge joining $a$ to $b$, as illustrated in Fig.~\ref{fig:tree_surgery}.

The component containing the path now consists of a cycle with two outward branches at each vertex, so it is isomorphic to $G$. This construction also applies when $\ell=1$ or $\ell=2$. The two detached branches joined by $\{a,b\}$ form an infinite connected acyclic graph in which every vertex has degree four, hence another copy of $T_4$.

Let $A_0$ and $B$ be the adjacency operators before and after these changes, acting on the same vertex set, and set $W=B-A_0$. Both adjacency matrices are real symmetric, so $W$ is self-adjoint. Only edges with endpoints in the finite set $K=\{v_1,v_\ell,a,b\}$ have changed. Therefore
\[
W_{xy}=0\qquad\text{whenever }x\notin K\text{ or }y\notin K.
\]
Lemma~\ref{lem:tree_finite_support_spectrum} with $q=3$ shows that $B$ has no singular continuous spectrum. Since $G$ is a connected component of the modified graph, extending any function on $G$ by zero preserves its spectral measure under the corresponding adjacency operators. Hence $A_G$ also has no singular continuous spectrum.
\end{proof}
\begin{theorem}
\label{thm:continuous_expander}
The limit graphexon operator $\mathcal{G}$ of the Gabber-Galil-Margulis expander graph sequence possesses a purely absolutely continuous real spectrum $[-\frac{\sqrt{3}}{2},\frac{\sqrt{3}}{2}]$ on the zero-mean subspace $L^2_0(\mathbb{T}^2, \nu)$. Furthermore, the spectral radius of $\mathcal{G}$ restricted to $L^2_0(\mathbb{T}^2, \nu)$ satisfies $ \|\mathcal{G}\|_{\mathcal{L}(L^2_0)} = \frac{\sqrt{3}}{2}$.
\end{theorem}

\begin{proof}
To analyze the spectral properties of the graphexon operator $\mathcal{G}$ on the continuous flat torus $\mathbb{T}^2$, we transfer the analysis to $\mathbb{Z}^2$ via the Fourier transform. We will first prove that $\mathcal{G}$ possesses a continuous spectrum. To prove this by contradiction, assume that there exists a real eigenvalue $\lambda$ and a non-zero eigenfunction $f \in L^2_0(\mathbb{T}^2, \nu)$ satisfying $\mathcal{G} f = \lambda f$.

Let $\{e_n\}_{n \in \mathbb{Z}^2 \setminus \{0\}}$ denote the standard Fourier basis for $L^2_0(\mathbb{T}^2, \nu)$ defined by $e_n(x) = \exp(2\pi i n \cdot x)$ for $x \in \mathbb{T}^2$. The function $f$ admits the Fourier series expansion $f(x) = \sum_{n \in \mathbb{Z}^2 \setminus \{0\}} c_n e_n(x)$ with $\sum_{n \in \mathbb{Z}^2 \setminus \{0\}} |c_n|^2 < \infty$. Let $c \in \ell^2(\mathbb{Z}^2 \setminus \{0\}) $ defined point-wise by $c(n)\triangleq c_n$ with abuse of notation.

Let $S_T = \{T_1, T_1^{-1}, T_2, T_2^{-1}\}$. Substituting the Fourier expansion into the eigenvalue equation and utilizing the algebraic identity $e_n(Tx) = e_{T^\top n}(x)$ yields
\[
\frac{1}{4}\sum_{T\in S_T}\sum_{n\in\mathbb{Z}^2\setminus\{0\}}c_n e_{T^\top n}(x)=\sum_{m\in\mathbb{Z}^2\setminus\{0\}}\lambda c_m e_m(x).
\]
Applying the change of variables $m = T^\top n$ and observing that the generator set $S_T$ is symmetric under inversion, matching the orthogonal basis coefficients reduces the functional equation to a discrete difference equation on the dual space $\mathbb{Z}^2 \setminus \{0\}$ given by
\begin{equation}\label{eq:fourier_difference}
\frac{1}{4}\sum_{T\in S_T}c_{T^\top m}=\lambda c_m.
\end{equation}
Define the operator $\mathcal{M}: \ell^2(\mathbb{Z}^2 \setminus \{0\}) \to \ell^2(\mathbb{Z}^2 \setminus \{0\})$ pointwise by $[\mathcal{M} c]_m \triangleq \frac{1}{4} \sum_{T \in S_T} c_{T^\top m}$. Then equation \eqref{eq:fourier_difference} is $\mathcal{M} c = \lambda c$.

Step 1. Let $S_T^\top = \{T^\top \mid T \in S_T\}$ and let $\tilde{\Gamma}^\top \subset \text{SL}(2, \mathbb{Z})$ denote the group generated by $S_T^\top$. Since $\tilde{\Gamma}^\top$ is a Zariski dense subgroup of $\text{SL}(2, \mathbb{R})$, its linear action on $\mathbb{R}^2$ admits no finite orbits other than the origin \cite{lubotzky1994discrete,zimmer2013ergodic}.

Let $\mathcal{I}$ be a countable index set and let $\{m_i\}_{i \in \mathcal{I}} \subset \mathbb{Z}^2 \setminus \{0\}$. The dual space $\mathbb{Z}^2 \setminus \{0\}$ is partitioned into disjoint infinite orbits $\mathcal{O}_{m_i} = \{ \gamma m_i \mid \gamma \in \tilde{\Gamma}^\top \}$ for $i \in \mathcal{I}$. Consequently, the Hilbert space admits the orthogonal direct sum decomposition
\[
\ell^2(\mathbb{Z}^2\setminus\{0\})=\bigoplus_{i\in\mathcal{I}}\ell^2(\mathcal{O}_{m_i}).
\]
Hence, any sequence $c \in \ell^2(\mathbb{Z}^2 \setminus \{0\})$ can be uniquely represented as $c = \sum_{i \in \mathcal{I}} \psi^{(i)}$, where $\psi^{(i)} \in \ell^2(\mathcal{O}_{m_i})$ is the orthogonal projection of $c$ onto the subspace $\ell^2(\mathcal{O}_{m_i})$. By the definition of $\mathcal{O}_{m_i}$, the subspace $\ell^2(\mathcal{O}_{m_i})$ is invariant under $\mathcal{M}$. Therefore, the eigenvalue equation $\mathcal{M} c = \lambda c$ holds if and only if $P_i \psi^{(i)} = \lambda \psi^{(i)}$ for all $i \in \mathcal{I}$, where $P_i: \ell^2(\mathcal{O}_{m_i}) \to \ell^2(\mathcal{O}_{m_i})$ denotes the restriction of the operator $\mathcal{M}$ to the invariant subspace $\ell^2(\mathcal{O}_{m_i})$.

For any fixed vector $m_0 \in \mathbb{Z}^2 \setminus \{0\}$, let $P \triangleq P_0$ and $\psi \triangleq \psi^0$ for simplicity. Then $P$ coincides with the transition operator of the Schreier graph $G = (\mathcal{O}_{m_0}, E)$ induced by the action of $\tilde{\Gamma}^\top$ with edge set $E=\{(m,\gamma m)\mid m \in \mathcal{O}_{m_0},\gamma \in S_T^\top\}$.

To determine the topological structure of $G$, Sanov's Theorem \cite{sanov1947property,de2000topics} establishes that $\tilde{\Gamma}^\top$ is isomorphic to the non-abelian free group $\mathbb{F}_2$, i.e., $\tilde{\Gamma}^\top \cong \mathbb{F}_2$ \footnote{The non-abelian free group $\mathbb{F}_2$ is the group consisting of all finite reduced words generated by two distinct elements and their inverses, subject to no algebraic relations other than the trivial cancellations of adjacent inverse elements.}. This implies that the Cayley graph of $\tilde{\Gamma}^\top$ with respect to $S_T^\top$ is a $4$-regular tree $T_4$ (see, e.g., \cite[Section 4.5]{lubotzky1994discrete}).

By definition, $G \cong T_4 / H$ where $H = \{ \gamma \in \tilde{\Gamma}^\top \mid \gamma m_0 = m_0 \}$ is the stabilizer subgroup. Matrices in $H$ possess an eigenvalue of $1$, implying they are either unipotent or the identity, and since they all fix $m_0$, they commute.

Given $\tilde{\Gamma}^\top \cong \mathbb{F}_2$, its abelian subgroups are cyclic (see, e.g., \cite{lyndon1977combinatorial}). Thus, $H$ is either trivial (containing only the identity matrix) or $H \cong \mathbb{Z}$ (generated by a single unipotent matrix). This indicates that, when $H$ is trivial, $G$ is a $4$-regular tree ($G \cong T_4$), and when $H \cong \mathbb{Z}$, $G$ possesses one finite core cycle with infinite trees attached to its vertices.

Step 2. To establish the absence of a point spectrum, we analyze the transition operator $P$ on the possible configurations of the orbit graph $G$. For the trivial stabilizer case where $G \cong T_4$, the operator corresponds to the simple random walk on a $4$-regular tree and therefore $G$ admits a continuous spectrum (see, e.g., \cite[Chapter 3]{figa1991harmonic}).

For the non-trivial configuration where $H \cong \mathbb{Z}$, the graph $G \cong T_4/H$ consists of a finite $2$-regular core cycle, with each vertex on the cycle serving as the root of an infinite outward-extending regular tree branch. Each cycle vertex has two attached rooted branches. It follows from \cite[Chapter 3]{figa1991harmonic} that the adjacency operator on a detached rooted branch has no $\ell^2$ eigenfunctions. Thus $\psi$ is radial, i.e., its value depends solely on the graph distance from the core cycle, on all extending tree branches.

Let $\psi_d\triangleq \psi(d)$ be the vector of these values at distance $d\geq0$ from the cycle, indexed by its vertices. Then $P\psi=\lambda\psi$ can be rewritten as
\[
\frac{1}{4}(\psi_{d-1}+3\psi_{d+1})=\lambda\psi_d,\qquad d\geq1.
\]
The $\ell^2$ integrability condition on the branching trees requires $\sum_d3^d\|\psi_d\|^2<\infty$. When $|\lambda|\leq\sqrt3/2$, both characteristic roots have modulus $1/\sqrt3$, with a repeated root at the endpoints, so no nonzero solution is square-summable. For $|\lambda|>\sqrt3/2$, the only admissible square-summable solution is $\psi_d=\alpha^d\psi_0$, subject to $|\alpha|<1/\sqrt3$. Substituting $\psi_d=\alpha^d\psi_0$ with $d=1$ establishes $4\lambda=3\alpha+\alpha^{-1}$.

Let $A_C$ denote the unnormalized adjacency operator of the core cycle. Evaluating the eigenvalue equation for the vertices on the cycle yields
\[
\frac{1}{4}(A_C\psi_0+2\psi_1)=\frac{1}{4}(A_C\psi_0+2\alpha\psi_0)=\lambda\psi_0.
\]
Substituting $4\lambda=3\alpha+\alpha^{-1}$ into the cycle equation reduces it to $A_C\psi_0=(\alpha+\alpha^{-1})\psi_0$. This forces $\alpha+\alpha^{-1}$ to be an eigenvalue of $A_C$. Since the core cycle is a finite $2$-regular graph, its eigenvalues, denoted by $2\mu$, are bounded in $[-2,2]$, which implies $\mu\in[-1,1]$. The algebraic relation $\alpha+\alpha^{-1}=2\mu$ yields the interface quadratic equation $\alpha^2-2\mu\alpha+1=0$. For any $\mu\in[-1,1]$, the discriminant is non-positive, ensuring the roots satisfy $|\alpha|=1$. This contradicts the normalizability constraint $|\alpha|<1/\sqrt3$.

Consequently, the operator $P$ admits no square-summable eigenfunctions on any orbit $\mathcal{O}_{m_0}$, forcing $c_m=0$ for all $m$. Therefore, the continuous operator $\mathcal{G}$ admits no point spectrum on $L^2_0(\mathbb{T}^2,\nu)$.

Step 3. To bound the continuous spectrum, recall from previous arguments that $G\cong T_4/H$ and $H$ is abelian. Then $H$ is amenable. By weak containment of the quasi-regular representation in the regular representation \cite[Theorem F.3.5]{bekka2008kazhdan}, the spectral radius of $P$ on $G$ does not exceed that on $T_4$. By Kesten's theorem \cite{kesten1958symmetric}, the spectral radius on $T_4$ is $\sqrt3/2$. Far from the finite cycle, the attached branches contain finite neighborhoods of arbitrarily large radius that agree with $T_4$. Placing finitely supported approximate eigenfunctions of the tree in these neighborhoods shows that every point of $[-\sqrt3/2,\sqrt3/2]$ belongs to the spectrum of $P$. Thus $\|P\|_{\mathcal{L}(\ell^2(\mathcal{O}_{m_0}))}=\sqrt3/2$ for every orbit, and
\[
\|\mathcal{G}\|_{\mathcal{L}(L^2_0)}=\sup_{m_0}\|P\|_{\mathcal{L}(\ell^2(\mathcal{O}_{m_0}))}=\frac{\sqrt3}{2}.
\]

Step 4. We now determine the spectral type of $P$. If the stabilizer is nontrivial, Step 1 identifies $G$ as a finite core cycle with two infinite rooted branches attached to each cycle vertex. Each vertex in an attached branch has three children. Lemma~\ref{lem:cycle_tree_surgery} therefore applies to the unnormalized adjacency operator $A_G=4P$ and shows that $A_G$ has no singular continuous spectrum. Step 2 excludes all $\ell^2$ eigenfunctions of $P$, and hence also of $A_G$. Thus $A_G$ has purely absolutely continuous spectrum, and $P$ also has purely absolutely continuous spectrum. If the stabilizer is trivial, then $G\cong T_4$, and the same conclusion follows from Lemma~\ref{lem:tree_finite_support_spectrum} with $q=3$ and $W=0$.

By the orbit decomposition in Step 1, $\mathcal{G}$ on $L^2_0(\mathbb{T}^2,\nu)$ is unitarily equivalent to the countable orthogonal direct sum of the operators $P_i$. Every $P_i$ has purely absolutely continuous spectrum, and the spectral measure of a vector in this direct sum is the sum of its orbit spectral measures. Hence $\mathcal{G}$ has purely absolutely continuous spectrum on $L^2_0(\mathbb{T}^2,\nu)$. Step 3 gives the stated spectral interval and norm.
\end{proof}

\section{Graphexon Mean Field Games on Expander Graphs}\label{sec:mfg}

To analyze the asymptotic behavior of massive populations interacting over the complex limit topologies, we formulate the graphexon mean field game on $(X, \nu)$ established in Section \ref{sec:weak_convergence}. 

Each agent is identified by a spatial network label $\alpha \in X$, which dictates its limit position, and only interacts with the neighborhood average defined by the operator $\mathcal{G}$. 

\subsection{Graphexon Mean Field Game Model}

Let $x: X \times [0, \infty) \to \mathbb{R}$ denote the state of the infinite population over time. For notational convenience, the shorthand $x_t^\alpha \triangleq x(\alpha, t)$ is utilized to represent the state of a representative agent located at the spatial network label $\alpha \in X$ at time $t \ge 0$. 

Let $\mu_t(\cdot; \alpha)$ denote the mean field representing the state distribution of the agent at label $\alpha$ at time $t$. The local mean state $m_t(\alpha)$ is defined as $ m_t(\alpha) \triangleq \int_{\mathbb{R}} y \mu_t(dy; \alpha)$.

Since the agents are distributed continuously over the compact metric space $X$, the global mean state $m_t(\cdot) \in L^2(X, \nu)$. The inner product on $L^2(X, \nu)$ is $\langle f, g \rangle_{L^2} \triangleq \int_{X} f(\alpha) g(\alpha) \nu(d\alpha)$ for any $f, g \in L^2(X, \nu)$.

Let the admissible control set $\mathcal{U}$ be the space of progressively measurable processes $u^\alpha$ satisfying the square-integrability condition $\mathbb{E} \int_0^\infty e^{-\gamma t} |u_t^\alpha|^2 dt < \infty$. 

The dynamics of the representative agent $\alpha$ are governed by the scalar stochastic differential equation
\begin{equation}
    d x_t^\alpha = (a x_t^\alpha + b u_t^\alpha + c [\mathcal{G} m_t](\alpha)) dt + \sigma dW_t^\alpha
\end{equation}
where $a \in \mathbb{R}$, $b, c \in \mathbb{R} \setminus \{0\}$, and $\sigma > 0$ are constant parameters. For each agent $\alpha$, $W_t^\alpha$ is an independent standard scalar Brownian motion.

The agents aim to minimize the infinite horizon discounted cost functional
\begin{equation} \label{eq:discounted_cost}
    J^\alpha(u^\alpha) = \frac{1}{2}\mathbb{E} \int_0^\infty e^{-\gamma t} \left[ q (x_t^\alpha - \eta [\mathcal{G} m_t](\alpha))^2 + r (u_t^\alpha)^2 \right] dt
\end{equation}
where $\gamma > 0$, $q > 0$, $r > 0$, and $\eta \in \mathbb{R}$.

Let $k \triangleq b^2/r$. Let $\Pi > 0$ be the unique positive stabilizing solution to the scalar algebraic Riccati equation
\begin{equation} \label{eq:baseline_are}
    2\Pi \left(a-\frac{\gamma}{2} \right) - k\Pi^2  + q = 0.
\end{equation}
Let $a_c \triangleq a - k \Pi$ denote the closed-loop parameter. We define the network coupling factor as $\psi \triangleq c\Pi - q\eta$.

\begin{lemma}
\label{lem:baseline_lqr}
The limit graphexon mean field game problem possesses a unique Nash solution if and only if the following coupled forward-backward ODE system admits a unique solution pair $(m_t, S_t)$ in $L^2(X, \nu)$
\begin{equation} \label{eq:abstract_system}
    \begin{cases}
        \dot{m}_t = (a_c \mathbf{I} + c \mathcal{G}) m_t - k S_t , \quad m_0(\alpha) = \int_{\mathbb{R}} y \mu_0(dy; \alpha)\\
        -\dot{S}_t = (a_c - \gamma) S_t + \psi \mathcal{G} m_t, \quad \lim_{t \to \infty} e^{-\gamma t} S_t = 0
    \end{cases}
\end{equation}
where $\mathbf{I}$ is the identity operator on $L^2(X, \nu)$. The initial mean state $m_0 \in L^2(X, \nu)$ is given. Furthermore, the optimal best response is given by the decentralized feedback law
\begin{equation}
    u_t^\alpha = -r^{-1}b (\Pi x_t^\alpha + S_t(\alpha)).
\end{equation}
\end{lemma}
\begin{proof}
The proof follows standard linear-quadratic mean field game literature (see e.g., \cite{caines2025cdc_mean, gao2023lqg}). Applying the maximum principle to the local agent's optimal control problem yields the individual adjoint equation of $S_t$.
\end{proof}
\subsection{Algebraic Riccati Equation and Mean State Stability}

Let $E$ be the projection-valued measure defined on the Borel subsets of the real line, associated with the self-adjoint operator $\mathcal{G}$. The graphexon operator is expressed as $\mathcal{G} = \int_{\sigma(\mathcal{G})} \lambda \, \mathrm{d}E(\lambda)$.

We seek a bounded linear operator $\mathbf{P}(\mathcal{G})$ on $L^2(X, \nu)$ defined by
\begin{equation}
    \mathbf{P}(\mathcal{G}) = \int_{\sigma(\mathcal{G})} p(\lambda) \, \mathrm{d}E(\lambda).
\end{equation}
Here $p(\lambda) \in \mathbb{R}$ acts as a spectral multiplier for each mode such that $S_t = \mathbf{P}(\mathcal{G}) m_t$.
Differentiating this linear ansatz and substituting it into the abstract system requires the following Scalar Algebraic Riccati Equation (SARE) to hold for all $\lambda \in \sigma(\mathcal{G})$
\begin{equation} \label{eq:nmare}
    -k p(\lambda)^2 + \left(2a_c - \gamma + \lambda c\right) p(\lambda) + \lambda \psi = 0.
\end{equation}

Define $a_\gamma \triangleq a_c - \gamma/2$. The discriminant associated with the SARE satisfies
\begin{equation}
    \Delta(\lambda) \triangleq c^2 \lambda^2 + 4(a_\gamma c + k\psi)\lambda + 4a_\gamma^2.
\end{equation}

\begin{lemma}
\label{lem:pointwise_sare}
Assume $c \neq 0$ and $a_c<0$. For any $\lambda \in \mathbb{R}$ satisfying $\Delta(\lambda) \ge 0$, the SARE admits a real-valued solution $p(\lambda)$ that coincides at $\lambda=0$ with the stable uncoupled solution $p(0)=0$. This branch satisfies
\begin{equation} \label{eq:p_lambda_branch}
    p(\lambda) = \frac{2a_\gamma + \lambda c + \sqrt{\Delta(\lambda)}}{2k}.
\end{equation}
The corresponding closed-loop generator satisfies
\begin{equation}
    \mathcal{A}_{cl}(\lambda) \triangleq a_c + \lambda c - k p(\lambda) = \frac{1}{2} (\gamma + \lambda c - \sqrt{\Delta(\lambda)}).
\end{equation}
\end{lemma}

\begin{proof}
Let $p_{\pm}(\lambda) = \frac{1}{2k} (2a_\gamma + \lambda c \pm \sqrt{\Delta(\lambda)})$. Substituting $2a_\gamma = 2a_c - \gamma$ into the closed-loop generator yields $ \mathcal{A}_{cl}^\pm(\lambda) = \frac{1}{2} (\gamma + \lambda c \mp \sqrt{\Delta(\lambda)})$.
Consider the negative branch $p_-(\lambda)$, which leads to $\mathcal{A}_{cl}^+(\lambda)$. Setting $\lambda = 0$, the stability condition $\mathcal{A}_{cl}^+(0) < 0$ requires $\gamma + \sqrt{4a_\gamma^2} < 0$. Since $\gamma > 0$ and $\sqrt{\Delta} \ge 0$, this is impossible. 
For the positive branch $p_+(\lambda)$, the generator is $\mathcal{A}_{cl}^-(\lambda)$. Setting $\lambda = 0$ provides $p_+(0) = (2a_\gamma + 2|a_\gamma|)/2k$. Since $a_\gamma < 0$, $p_+(0) = 0$, which is the unique solution consistent with the unperturbed system.
\end{proof}

Let $\lambda = 1$ denote the isolated principal eigenvalue of $\mathcal{G}$ associated with the global mean state $\bar{m}_t$ defined by
\[
    \bar{m}_t \triangleq \int_{X} m_t(\alpha) \nu(d\alpha).
\]
The spatial deviation field is defined by $e_t \triangleq m_t-\bar{m}_t\mathbf{1}$.
To guarantee the uncoupled dynamics are stable, it is assumed that $a_c < 0$. Define $\Theta \triangleq a_c(a_c - \gamma) > 0$ and the function $\Phi(c) \triangleq (\gamma - a)c + k q \eta$. The mean state stability manifold is defined by $\mathcal{S}_{1} \triangleq \{ c \in \mathbb{R} \mid \mathcal{A}_{cl}(1) < 0 \}$. Define the mean state bifurcation threshold $c^* \triangleq \frac{\Theta - k q \eta}{\gamma - a}$.

\begin{theorem}
Assume $\Delta(1) \ge 0$. The mean state stability manifold $\mathcal{S}_{1}$, which characterizes the uniform stability of the global average $\bar{m}_t$, is completely independent of the sparse topological spectrum and is determined by the following domains.
\begin{enumerate}
    \item If $\gamma > a$, then $\mathcal{S}_{1} = \left(-\infty, \, \max\left(-\gamma, c^*\right)\right)$.
    \item If $\gamma < a$, then $\mathcal{S}_{1} = (-\infty, -\gamma) \cup (c^*, \infty)$.
    \item If $\gamma = a$, the stability depends on the parameter gap
    \begin{enumerate}
        \item If $k q \eta < \Theta$, then $\mathcal{S}_{1} = \mathbb{R}$.
        \item If $k q \eta \ge \Theta$, then $\mathcal{S}_{1} = (-\infty, -\gamma)$.
    \end{enumerate}
\end{enumerate}
\end{theorem}

\begin{proof}
The inequality condition $\mathcal{A}_{cl}(1) < 0$ reduces to satisfying either $c < -\gamma$ or $\Phi(c) < \Theta$. Applying this logic condition across the three cases defined by the sign of $\gamma - a$ directly yields the specified disjoint intervals.
\end{proof}

\subsection{Spectral Gap and Turing-Type Instability Boundaries}

This subsection investigates the boundaries of the spatial deviation field, which relies on the spectral gap of the underlying sparse network. Let $\Sigma_0 \triangleq \sigma(\mathcal{G}|_{L^2_0(X,\nu)})=[-\rho,\rho]$ denote the continuous spectral domain of $\mathcal{G}$ for some $0<\rho<1$. Note that $\rho=\frac{\sqrt{3}}{2}$ in the 2-torus setting discussed in Section \ref{sec:margulis}.

\begin{theorem}
\label{thm:existence_and_branch}
Assume $c \neq 0$. The SARE discriminant satisfies $\Delta(\lambda) \ge 0$ for all $\lambda \in [-\rho, \rho]$ if and only if the following two algebraic conditions hold.
\begin{enumerate}
    \item $ c^2 \rho^2 + 4a_\gamma^2 \ge 4\rho |a_\gamma c + k\psi|$;
    \item Let $\lambda^* = -2(a_\gamma c + k\psi)/c^2$. If $|\lambda^*| \le \rho$, then $k\psi(2a_\gamma c + k\psi) \le 0$.
\end{enumerate}
\end{theorem}

\begin{proof}
As $c \neq 0$, $\Delta(\lambda)$ is a convex parabola. Non-negativity on the compact interval $[-\rho, \rho]$ is equivalent to the non-negativity of its values at the boundaries and its vertex (if contained in the interval).
$\Delta(\pm \rho) \ge 0$ is equivalent to $c^2 \rho^2 + 4a_\gamma^2 \ge 4\rho |a_\gamma c + k\psi|$.
The global minimum is $\Delta(\lambda^*)$ at $\lambda^* = -2(a_\gamma c + k\psi)/c^2$. If $|\lambda^*| \le \rho$, the condition $\Delta(\lambda^*) \ge 0$ simplifies to its own discriminant being non-positive, satisfying $16(a_\gamma c + k\psi)^2 - 16 c^2 a_\gamma^2 \le 0$, which factors into $k\psi(2a_\gamma c + k\psi) \le 0$.
\end{proof}

Define $L: \mathbb{R} \to \mathbb{R}$ by
\begin{equation}
    L(\lambda) \triangleq ((\gamma - a)c + k q \eta)\lambda - a_c(a_c - \gamma).
\end{equation}
It follows that $\mathcal{A}_{cl}(\lambda)<0$ is equivalent to $L(\lambda)<0$.

\begin{proposition}
\label{thm:complete_stability}
Let $\sigma(\mathcal{G})=[-\rho, \rho] \cup \{1\}$. Assume the algebraic conditions in Theorem \ref{thm:existence_and_branch} and $\Delta(1) \ge 0$ hold. Define the non-negative half-space $\Omega \triangleq \{ \lambda \in \sigma(\mathcal{G}) \mid \gamma + \lambda c \ge 0 \}$, and let $\mathcal{E} \triangleq \{1, -\rho, \rho, -\gamma/c\} \cap \Omega$.
Then the infinite-dimensional mean field system is globally asymptotically stable, that is, $\lim_{t \to \infty} \|m_t\|_{L^2} = 0$, if and only if $a_c < 0$ and, for every point $\lambda_e \in \mathcal{E}$, $L(\lambda_e)<0$.
\end{proposition}

\begin{proof}
Global asymptotic stability requires the closed-loop generator to satisfy $\mathcal{A}_{cl}(\lambda) < 0$ for all spectral modes $\lambda \in \sigma(\mathcal{G})$. By Lemma \ref{lem:pointwise_sare}, this condition is equivalent to $\gamma + \lambda c < \sqrt{\Delta(\lambda)}$.

For any $\lambda \in \sigma(\mathcal{G}) \setminus \Omega$, the term $\gamma + \lambda c<0$ and $\mathcal{A}_{cl}(\lambda)$ is clearly less than zero.

For any $\lambda \in \Omega$, both sides of the inequality are non-negative. One can verify that $\mathcal{A}_{cl}(\lambda)<0$ is equivalent to $L(\lambda)<0$. The supremum of $L$ is attained at the extreme points of $\Omega$, which are contained in the finite set $\mathcal{E}$. Therefore, $L(\lambda) < 0$ holds for all $\lambda \in \Omega$ if and only if $L(\lambda_e) < 0$ for all $\lambda_e \in \mathcal{E}$.

Since $\gamma > 0$, the base mode $\lambda = 0$ is contained in $\Omega$. The condition $L(0) = -a_c(a_c - \gamma) < 0$ requires $a_c < 0$ due to $a_\gamma=a_c-\gamma/2<0$, completing the proof.
\end{proof}

\begin{definition}[Turing-Type Topological Instability]
\label{def:turing_instability}
The graphexon mean field game system undergoes a \textit{Turing-type topological instability} if the uncoupled local dynamics (i.e., $c=0$) are stable, the average state $\bar{m}_t$ remains stable, i.e., $\mathcal{A}_{cl}(1) < 0$, whereas $e_t$ diverges such that $ \sup_{\lambda \in \Sigma_0} \mathcal{A}_{cl}(\lambda) > 0$.
\end{definition}

To characterize the spatial deviation instability on the continuous spectrum boundary, the bifurcation thresholds are defined by
\[
    c^+ \triangleq \frac{\Theta/\rho - k q \eta}{\gamma - a}, \quad c^- \triangleq \frac{-\Theta/\rho - k q \eta}{\gamma - a}.
\]

\begin{lemma}
\label{lem:monotone_generator}
The closed-loop generator $\mathcal{A}_{cl}(\lambda)$ is weakly monotonic with respect to $\lambda\in[-\rho, \rho]$. Consequently, the Turing-type instability condition is equivalent to $ \max\{\mathcal{A}_{cl}(\rho), \mathcal{A}_{cl}(-\rho)\} > 0$.
\end{lemma}

\begin{proof}
One can verify that the equation $\frac{d}{d\lambda}\mathcal{A}_{cl}(\lambda) = 0$ is equivalent to $k\psi(2a_\gamma c + k\psi) = 0$. Thus, the derivative $\frac{d}{d\lambda}\mathcal{A}_{cl}(\lambda)$ maintains the same sign across the interior of the spectral interval $(-\rho, \rho)$. Therefore, the supremum of $\mathcal{A}_{cl}(\lambda)$ on $[-\rho, \rho]$ is attained at the spectral boundaries $\lambda \in \{-\rho, \rho\}$, establishing the equivalence.
\end{proof}

Define the spatial instability manifold $\mathcal{I} \triangleq \mathcal{I}_0 \setminus \{0\}$ where $\mathcal{I}_0 \triangleq \{ c \in \mathbb{R} \mid \max\{\mathcal{A}_{cl}(\rho), \mathcal{A}_{cl}(-\rho)\} > 0 \}$. 

\begin{theorem}\label{thm:exact_bifurcation_interval}
Assume the algebraic conditions in Theorem \ref{thm:existence_and_branch}. Then $\mathcal{I}_0$ is determined by the following domains.
\begin{enumerate}
    \item If $\gamma > a$, then $ \mathcal{I}_0 = \left( -\infty, \, \min\left(\frac{\gamma}{\rho}, c^-\right) \right) \cup \left( \max\left(-\frac{\gamma}{\rho}, c^+\right), \, \infty \right)$.
    
    \item If $\gamma < a$, then $\mathcal{I}_0 = \left(-\frac{\gamma}{\rho}, \, c^+ \right) \cup \left( c^-, \, \frac{\gamma}{\rho}\right)$.
    \item If $\gamma = a$ and
    \begin{enumerate}
        \item $\rho k |q \eta| \le \Theta$, then $\mathcal{I}_0 = \emptyset$.
        \item $\rho k q \eta > \Theta$, then $\mathcal{I}_0 = \left(-\frac{\gamma}{\rho}, \infty\right)$.
        \item $\rho k q \eta < -\Theta$, then $\mathcal{I}_0 = \left(-\infty, \frac{\gamma}{\rho}\right)$.
    \end{enumerate}
\end{enumerate}
Consequently, a Turing-type instability occurs when the coupling parameter satisfies $c \in \mathcal{S}_{1} \cap \mathcal{I}$ and $a_c<0$.
\end{theorem}

\begin{proof}
By Lemma \ref{lem:monotone_generator}, spatial instability is equivalent to $\mathcal{A}_{cl}(\rho) > 0$ or $\mathcal{A}_{cl}(-\rho) > 0$. 

The condition $\mathcal{A}_{cl}(\rho) > 0$ reduces to $\gamma + \rho c > \sqrt{\Delta(\rho)}$, which is satisfied if and only if $c > -\gamma/\rho$ and $\Phi(c) > \Theta/\rho$. By symmetry, $\mathcal{A}_{cl}(-\rho) > 0$ is satisfied if and only if $c < \gamma/\rho$ and $-\Phi(c) > \Theta/\rho$.

Case 1 ($\gamma > a$). The union of the solution sets directly yields the two disjoint rays for $\mathcal{I}_0$.

Case 2 ($\gamma < a$). The conditions resolve to $c \in (-\gamma/\rho, c^+)$ and $c \in (c^-, \gamma/\rho)$, respectively. Since $\Theta > 0$ and $\gamma - a < 0$, it follows that $c^+ - c^- = \frac{2\Theta/\rho}{\gamma - a} < 0$, which guarantees $c^+ < c^-$. The union of these domains generates the stated disjoint bounded intervals.

Case 3 ($\gamma = a$). Here, $\Phi(c) = k q \eta$. The quadratic constraints simplify to $\rho k q \eta > \Theta$ and $-\rho k q \eta > \Theta$. Because $\Theta > 0$, at most one inequality holds. 
If $\rho k |q \eta| \le \Theta$, both inequalities fail, yielding $\mathcal{I}_0 = \emptyset$.
If $\rho k q \eta > \Theta$, the first inequality holds, restricting instability to $c > -\gamma/\rho$, while the second fails unconditionally, yielding $\mathcal{I}_0 = \left(-\frac{\gamma}{\rho}, \infty\right)$.
If $\rho k q \eta < -\Theta$, the second inequality holds, restricting instability to $c < \gamma/\rho$, while the first fails unconditionally, yielding $\mathcal{I}_0 = \left(-\infty, \frac{\gamma}{\rho}\right)$.
\end{proof}

\subsection{Continuous Spectrum and Spatial Fragmentation}

To discuss the specific properties brought by the purely continuous spectrum of $\mathcal{G}$, we restrict our analysis to the graphexon framework established in Section \ref{sec:margulis}. Unlike classical Graphon MFGs where the dense integral operator is compact and Turing instability leads to strongly convergent stationary spatial patterns, the graphexon operator is non-compact. Consequently, the Turing-type topological instability does not yield stable macroscopic patterns. To isolate the spatial profile of the instability from its exponential growth, we define the normalized deviation field
\[
    \tilde{e}_t(\alpha) \triangleq \frac{e_t(\alpha)}{ \|e_t\|_{L^2}}.
\]

In classical models with compact operators, the existence of discrete $L^2$-eigenfunctions dictates that $\tilde{e}_t$ converges strongly to a stationary spatial pattern (i.e., $\tilde{e}_t \to \phi$ in $L^2$ for some eigenfunction $\phi$). In contrast, the graphexon operator $\mathcal{G}$ in Section \ref{sec:margulis} possesses a purely continuous spectrum on $L^2_0(\mathbb{T}^2, \nu)$ and admits no discrete eigenfunctions. The following theorem demonstrates that $\tilde{e}_t$ converges weakly to zero.

Let $\mu_{e_0}(\cdot) \triangleq \langle E(\cdot)e_0, e_0 \rangle_{L^2}$ denote the scalar spectral measure associated with the initial state $e_0 \in L^2_0(\mathbb{T}^2, \nu)$. Assume the system exhibits a Turing-type instability. Let $\omega_{\max} \triangleq \sup_{\lambda \in \text{supp}(\mu_{e_0})} \mathcal{A}_{cl}(\lambda)$ denote the maximum growth rate relative to the initial state, and assume $\omega_{\max} > 0$. Let $\lambda^* \in \text{supp}(\mu_{e_0})$ be the unique dominant spectral mode satisfying $\mathcal{A}_{cl}(\lambda^*) = \omega_{\max}$.

\begin{theorem}
For any test function $\phi \in L^2_0(\mathbb{T}^2, \nu)$, the normalized deviation field converges weakly to zero, satisfying
\[
    \lim_{t \to \infty} \langle \tilde{e}_t, \phi \rangle_{L^2} = 0.
\]
\end{theorem}

\begin{proof}
Let $E$ denote the projection-valued measure associated with the self-adjoint Margulis operator $\mathcal{G}$. Because $\mathcal{G}$ possesses a purely absolutely continuous spectrum on $L^2_0(\mathbb{T}^2, \nu)$, the corresponding scalar spectral measure $\mu_\phi(\cdot) \triangleq \langle E(\cdot)\phi, \phi \rangle_{L^2}$ is non-atomic for any test function $\phi \in L^2_0(\mathbb{T}^2, \nu)$.

Let $I_\delta = (\lambda^* - \delta, \lambda^* + \delta) \cap \Sigma_0$ be a $\delta$-neighborhood centered at the dominant spectral mode. Applying the orthogonal decomposition of the projection-valued measure yields the inner product relation
\[
    \langle \tilde{e}_t, \phi \rangle_{L^2} = \langle \tilde{e}_t, E(I_\delta)\phi \rangle_{L^2} + \langle \tilde{e}_t, E(I_\delta^c)\phi \rangle_{L^2}.
\]
By definition of $\lambda^* \in \text{supp}(\mu_{e_0})$, the spectral measure satisfies $\| E(I_\delta) e_0 \|_{L^2} > 0$ for any $\delta > 0$. The energy component distributed over the complement set $I_\delta^c$ grows at an exponential rate bounded away from the maximal rate $\omega_{\max}$. Consequently, the total norm $\|e_t\|_{L^2}$ is exponentially dominated by the modes within $I_\delta$. Under the continuous normalization factor $\|e_t\|_{L^2}$, the inner product term corresponding to $I_\delta^c$ vanishes asymptotically as $t \to \infty$.

By the Cauchy-Schwarz inequality, the absolute value of the limit supremum is bounded by $\|E(I_\delta)\phi\|_{L^2}$. Because the spectral measure is non-atomic, taking the limits iteratively yields $\lim_{\delta \to 0} \|E(I_\delta)\phi\|_{L^2} = \|E(\{\lambda^*\})\phi\|_{L^2} = 0$. This establishes the weak convergence of the normalized field to zero.
\end{proof}


\section{Numerical Simulation}\label{sec:simulation}

In this section, we simulate the spatial evolution of the graphexon MFG on a $40 \times 40$ discrete Margulis expander graph ($N = 40$ and $M_N = 1600$). The fixed system parameters are $a = -1$, $b = 1$, $q = 2$, $r = 1$, $\gamma = 0.5$, $\eta = 0$, and $\sigma = 2$. Based on Theorem \ref{thm:complete_stability} and Theorem \ref{thm:exact_bifurcation_interval}, the theoretical Turing bifurcation thresholds are $c^-\approx-2.69$ and $c^+\approx2.69$ given the limit spectral radius $\rho = \sqrt{3}/2$. Accordingly, we compare two drift coupling regimes: $c_{stable} = 1.0$ (asymptotic stability) and $c_{turing} = -5$ (Turing instability). The initial states are generated using a low-frequency sine wave with Gaussian noise, and the closed-loop dynamics are computed over $t \in [0, 15]$.
The mean-mode threshold is $c^*\approx2.33$. Thus, $c_{stable}=1$ is globally stable, whereas $c_{turing}=-5<c^-$ retains mean-mode stability but destabilizes the spatial modes, with $\mathcal{A}_{cl}(1)\approx-6.26$ and $\mathcal{A}_{cl}(-\rho)\approx0.73$.

\begin{figure}[htbp]
    \centering
    \begin{subfigure}{0.48\textwidth}
        \centering
        \includegraphics[width=0.8\linewidth]{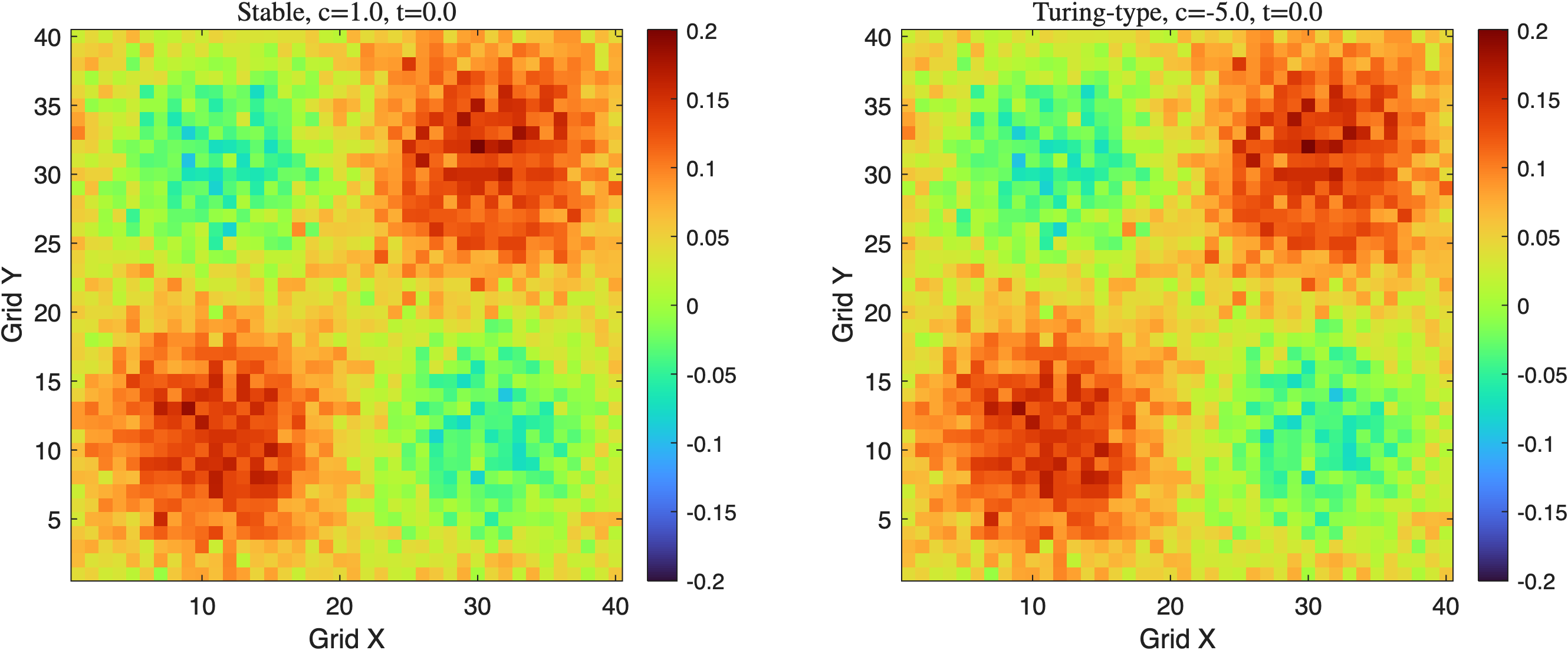}
        \caption{Spatial Evolution ($t = 0$)}
        \label{fig:spatial_t0}
    \end{subfigure}
    \hfill
    \begin{subfigure}{0.48\textwidth}
        \centering
        \includegraphics[width=0.8\linewidth]{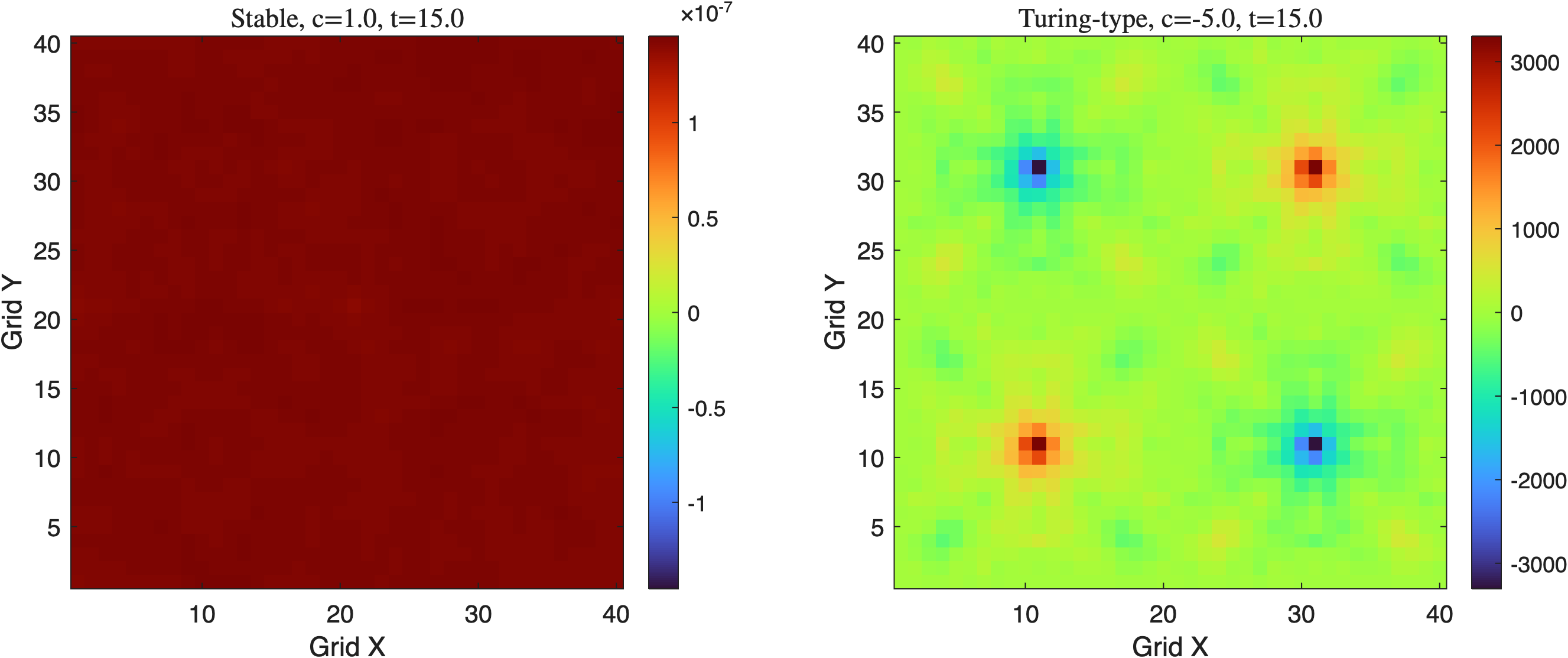}
        \caption{Spatial Evolution ($t = 15$)}
        \label{fig:spatial_t3}
    \end{subfigure}
    \caption{Spatial distribution of agent states. Under $c_{stable}$, the system reaches a homogeneous spatial consensus. Under $c_{turing}$, the field fragments into high-frequency spatial oscillations.}
    \label{fig:spatial_patterns}
\end{figure}

As illustrated in Figure \ref{fig:spatial_patterns}, under $c_{stable}$, the initial spatial perturbations decay, leading to a flat, homogeneous consensus. Conversely, under $c_{turing}$, the spatial field diverges into a fragmented pattern with rapid high-frequency oscillations. This macroscopic fragmentation is consistent with the theoretical prediction that the unstable modes of the limit graphexon operator lack smooth $L^2$ eigenfunctions.

\section{Conclusion}\label{sec:conclusion}

This paper formulates and solves infinite-horizon discounted MFGs on expander graphs. The primary contribution extends Lubotzky's algebraic framework for finite expander graphs to the continuous graphexon limit. The weak convergence of empirical graph measures to a limit graphexon measure is established and a scalar LQG MFG on the limit space is formulated. Furthermore, the stability boundaries and Turing-type topological instability is analyzed.
The uniform spectral gap of expander graph sequence constructed in this paper is retained in the limit graph object. This gives rise a critical value for the coupling parameter $c$ that induces a system bifurcation.

Future research could incorporate the Laplace-Beltrami operator—induced by microscopic spatial Brownian motion—into the macroscopic forward-backward equations. This resulting reaction-diffusion-expander system will introduce high-frequency dissipation, maintaining spatial regularity for the large-scale multi-agent MFG systems.

\bibliography{ref}
\end{document}